\documentclass[10pt]{amsart}
\usepackage[utf8x]{inputenc}
\usepackage{dsfont}
\usepackage{amsmath}
\usepackage{amssymb}
\usepackage{graphicx}
\usepackage{amssymb}
\usepackage{amsfonts}
\usepackage{soul}
\usepackage{mathpazo}
\usepackage{newunicodechar}
\newunicodechar{：}{:}
\usepackage{color}
\usepackage{yfonts}
\usepackage{paralist}
\usepackage{stmaryrd}
\usepackage{amsxtra}
\usepackage{latexsym,mathrsfs}
\usepackage{verbatim}
\usepackage{mathabx}

\usepackage{epsfig, enumerate}
\usepackage{color}
\usepackage{hyperref}

\newtheorem{theorem}{Theorem}
\newtheorem*{lemma*}{Claim}
\newtheorem{lemma}[theorem]{Lemma}
\newtheorem{coro}[theorem]{Corollary}

\newtheorem{prop}[theorem]{Proposition}

 \theoremstyle{definition}\newtheorem{definition}[theorem]{Definition}

\numberwithin{theorem}{section}
\numberwithin{equation}{section}
\definecolor{turquoise}{cmyk}{0.65,0,0.1,0.1}
\definecolor{purple}{rgb}{0.65,0,0.65}
\definecolor{green}{rgb}{0, 0.5, 0}
\definecolor{blue}{rgb}{0, 0, 1}
\definecolor{orange}{rgb}{0.8, 0.6, 0.2}
\definecolor{red}{rgb}{0.8, 0.2, 0.2}
\definecolor{brown}{rgb}{0.5, 0.16, 0.16}

\title[MC Component]{Large Monochromatic Components in Colored Random Graphs}

\author{Xiao-Chuan Liu}
\address[Liu]{Departamento de Matemática,
 Universidade Federal de Pernambuco,
	Avenida Jornalista Aníbal Fernandes, Cidade Universitária, Recife, Brazil}
\email{xiaochuan.liu@ufpe.br}

\author{Xu Yang}
\address[Yang]{Instituto de Computação,  Universidade Federal de Alagoas,
	Av. Lourival Melo Mota, S/N, Maceió, Brazil}
\email{yang@ic.ufal.br}

\begin{document}
\maketitle

\begin{abstract}
We study the size of the largest monochromatic connected component that must appear in any edge-coloring of a random graph.

Let $G\sim G(n,p)$ with $p\gg 1/n$ and $p=o(1)$, and write $np=he^h$. We show that, with high probability, every $2$-edge-coloring of $G$ contains a monochromatic connected component of order at least $n-\Theta(ne^{-h})$. Moreover, we construct colorings showing that this bound is best possible up to constant factors.

We extend this result to three colors: for $p\gg 1/n$ and $p=o(1)$, with high probability every $3$-edge-coloring of $G$ contains a monochromatic connected component of size at least $\frac{n}{2}-\Theta(1/p)$, and this estimate is again tight up to constant factors.

In the bipartite setting $G\sim G(n,n,p)$, under the same assumptions on $p$, we prove an analogous statement: with high probability, every $2$-edge-coloring contains two monochromatic components whose union covers all but $\Theta(ne^{-h})$ vertices, and this bound is asymptotically sharp.

Our approach is elementary and is based on analyzing large connected structures across suitably balanced vertex partitions.
\end{abstract}

\section{Introduction}
Understanding the structures that must appear in edge-colorings of graphs is a central theme in combinatorics, with deep connections to Ramsey theory, extremal graph theory, and probabilistic methods. A classical line of research investigates how large a monochromatic structure must exist under arbitrary edge-colorings, with particular emphasis on connected subgraphs, which capture global organization in graphs.
In deterministic settings, the study of large monochromatic connected components has a long history. A fundamental problem, posed and systematically studied by Gy\'arf\'as ~\cite{Gyarfas77,Gyarfas87}, asks for the order of the largest monochromatic connected component that must appear in every $r$-edge-coloring of the complete graph $K_n$. This problem has become a cornerstone in the study of monochromatic structures in edge-colored graphs, inspiring a broad line of research on unavoidable configurations and playing an important role in the development of size-Ramsey theory. Formally, for a graph $G$, let $m_r(G)$ denote the minimum possible order of the largest monochromatic connected component over all $r$-edge-colorings of $G$. 

Gy\'arf\'as \cite{Gyarfas77} showed that for two colorings the largest monochromatic component must span the entire graph, that is, \(m_2(K_n)=n\), while for three colorings it must have order at least about half of the vertices, \(m_3(K_n)=\lceil n/2\rceil\). The latter bound is tight, as witnessed by a simple extremal construction obtained by partitioning the vertex set into four nearly equal parts. Partition the vertex set of \(K_n\) into four parts \(A_1,A_2,B_1,B_2\) whose sizes differ by at most one. Let \(A=A_1\cup A_2\) and \(B=B_1\cup B_2\). Color all edges inside \(A\) and inside \(B\) green, all edges between \(A_1\) and \(B_1\) and between \(A_2\) and \(B_2\) blue, and all edges between \(A_1\) and \(B_2\) and between \(A_2\) and \(B_1\) red. In this coloring, every monochromatic connected component has order at most \(\lceil n/2\rceil\), and this bound is attained. 

More generally, results of Gyárfás~\cite{Gyarfas87} and Füredi~\cite{Furedi1981} imply  that in every \(r\)-edge-coloring of the complete graph \(K_n\), 
there exists a monochromatic connected component of order at least 
$\Big\lceil \frac{n}{r-1} \Big\rceil$. 
This lower bound was later shown to follow directly from a simple counting argument 
(see Theorem~11 of~\cite{LiuMorrisPrince2009}).  
Motivated by this, the following exact formula for \(m_r(K_n)\) 
was conjectured (see, e.g.,~\cite{Gyarfas87,GyarfasFurediSimonyi93}):
\begin{equation}
m_r(K_n)=\Big\lceil \frac{n}{r-1}\Big\rceil.
\end{equation}
The conjecture is known to hold whenever $r-1$ is a prime power. In these cases, the extremal colorings can be obtained from affine-plane constructions of order  \(r-1\).

In contrast, the random graph model provides a natural framework for studying how randomness influences such unavoidable structures. Random graphs have long served as a testing ground for understanding threshold phenomena and the typical behavior of extremal parameters.
Inspired by their deterministic counterparts, the study of monochromatic structures in random graphs has attracted considerable attention in recent years. A foundational result in this direction, proved by Bal and DeBiasio~\cite{BalDeBiasio17} and independently by Dudek and Prałat~\cite{DudekPralat2017},
shows that whenever $np\to\infty$, every $r$-edge-coloring of a random graph typically contains a monochromatic connected component of order close to $n$ divided by $(r-1)$.

\begin{theorem}[Theorem~1.8 of~\cite{BalDeBiasio17}, Theorem~5.3 of~\cite{DudekPralat2017}]
\label{thm:Bal-DeBiasio}
Let $r \ge 2$ and suppose $p \gg 1/n$.
Then, with high probability, every 
$r$--edge--coloring 
of the random graph $G\sim G(n,p)$ contains a monochromatic connected component of order at least 
\begin{equation}
\frac{n}{r-1} - o(n).
\end{equation}
\end{theorem}

However, their approach relies on the sparse regularity method, developed by Kohayakawa~\cite{Kohayakawa1997} and Rödl (see, e.g.,~\cite{RodlSchacht2007}), which involves substantial technical machinery. Moreover, it yields only an unspecified $o(n)$ error term, without determining its correct order.

In this paper, we determine this error term up to constant factors for the cases of two and three colors. In particular, we obtain sharp bounds on the size of the largest monochromatic connected component that must appear in every coloring of a random graph. Our results reveal that, in contrast to the deterministic setting, the behavior of this parameter depends delicately on the edge probability.

Our approach is entirely elementary and avoids the use of the sparse regularity method. Instead, it is based on a direct analysis of large connected structures arising from suitably balanced vertex partitions.

\begin{theorem}\label{thm:mono-sharp}
Let $G\sim G(n,p)$, where $p\gg 1/n$ and $p=o(1)$, and write $np=he^h$. Then there exist absolute constants $c,C>0$ such that, with high probability, the following statements hold.
\begin{enumerate}
\item[(i)] \textbf{Giant monochromatic component.} Every $2$-edge-coloring of $G$ contains a monochromatic connected component of order at least
\begin{equation}
n-Cne^{-h}.
\end{equation}
\item[(ii)] \textbf{Extremal coloring.} There exists a $2$-edge-coloring of $G$ such that every monochromatic connected component has order at most
\begin{equation}
n-cne^{-h}.
\end{equation}
\end{enumerate}
\end{theorem}

\noindent\textbf{Remark.}
For the random graph $G(n,p)$ with $f=np=\omega(1)$ and
$f\leq \log n-\omega(1)$, there is a unique \emph{giant connected component} $C_{\text{giant}}$ whose order satisfies
\begin{equation}
|C_{\text{giant}}|=n-(1+o(1))ne^{-f}.
\end{equation}
See, for example,~\cite{Bollobas2001,FriezeKaronski2015} for standard references on the emergence of the giant component in random graphs. This should be compared with the order of the ``giant monochromatic component" obtained in our theorem.

\begin{theorem}\label{thm:three-color}
Let $G\sim G(n,p)$, where $p\gg 1/n$ and $p=o(1)$. Then there exist absolute constants $c,C>0$ such that, with high probability, the following statements hold.
\begin{enumerate}
\item[(i)] \textbf{Large monochromatic component.} Every $3$-edge-coloring of $G$ contains a monochromatic connected component of order at least
\begin{equation}\label{eq:3color-lower}
\frac{n}{2}-\frac{C}{p}.
\end{equation}
\item[(ii)] \textbf{Extremal coloring.} There exists a $3$-edge-coloring of $G$ such that every monochromatic connected component has order at most
\begin{equation}\label{eq:3color-upper}
\frac{n}{2}-\frac{c}{p}.
\end{equation}
\end{enumerate}
\end{theorem}

\noindent\textbf{Remark.}
Note that in the cases $r=2$ and $r=3$, the quantities $n - m_2(G(n,p))$ and $\frac{n}{2} - m_3(G(n,p))$ have different magnitudes. 
The former is $\Theta(n e^{-h})$ when $np = h e^{h}$, 
while the latter is $\Theta(\frac{1}{p})$. 
We believe that for $r \ge 4$, this latter magnitude $\Theta(\frac{1}{p})$ is the correct one for $\frac{n}{r-1} - m_r(G(n,p))$. 
This leads to the following conjecture.

\medskip
\noindent\textbf{Conjecture.} 
Let $r\ge 4$ be a fixed integer, and suppose $G\sim G(n,p)$, where $p\gg 1/n$ and $p=o(1)$. Then, with high probability,
\begin{equation}
\frac{n}{r-1}-m_r\bigl(G(n,p)\bigr)=
\Theta\Bigl(\frac{1}{p}\Bigr).
\end{equation}

Using a similar method, we can also obtain an analogous result for the bipartite random graph model $G(n,n,p)$.

\begin{theorem}\label{thm:bipartite-mono}
Let $G\sim G(n,n,p)$, where $p\gg 1/n$ and $p=o(1)$, and write $np=he^h$. Then there exist absolute constants $c,C>0$ such that, with high probability, the following statements hold.
\begin{enumerate}
\item[(i)] \textbf{Two large monochromatic components.} Every $2$-edge-coloring of $G$ contains two monochromatic connected components $C_1$ and $C_2$ such that
\begin{equation}
|C_1\cup C_2|\geq 2n-Cne^{-h}.
\end{equation}
\item[(ii)] \textbf{Extremal coloring.} There exists a $2$-edge-coloring of $G$ such that every pair of monochromatic connected components $C_1$ and $C_2$ satisfies
\begin{equation}
|C_1\cup C_2|\leq 2n-cne^{-h}.
\end{equation}
\end{enumerate}
\end{theorem}

Before closing the introduction, let us comment on some generalities. 
For standard notation in graph theory and probabilistic combinatorics, as well as certain basic tools, such as Chernoff's bounds, 
we refer to~\cite{alon2016probabilistic} for convenient reference, and will not cite each time we apply it. 
Also, throughout the paper we will often ignore floors and ceilings whenever they are not essential.
For some more specific (still standard) properties of random graphs, we have chosen to present the relevant results within the corresponding arguments rather than collecting them in a preliminary section, in the hope of improving the overall readability of the proofs.

Below, in Section~\ref{extremal_colouring} we prove part~(ii) of all three theorems. 
Section~\ref{large_components} establishes a technical proposition and applies it to prove part~(i) of Theorems~\ref{thm:mono-sharp} and~\ref{thm:bipartite-mono}. 
Finally, Section~\ref{ioftheorem3} is devoted to the proof of part~(i) of Theorem~\ref{thm:three-color}.

\section{Extremal colorings}\label{extremal_colouring}

In this section, we will
focus on constructing the extremal colorings for all three theorems.

\begin{proof}[Proof of (ii) of Theorem~\ref{thm:mono-sharp}.]
Assume $f=np=h e^{h}$. 
Set $r=\tfrac{1}{2}n e^{-h}$.
Fix a vertex subset $R\subseteq V$ with $|R|=r$ and put $A=V \setminus R$.
In the graph $G$, we color each present edge with both endpoints in $A$ or both in $R$ blue, and each present edge between $A$ and $R$ red.

Blue edges appear only inside $A$ or $R$, so every blue component is contained in one of them and thus has size at most $|A|=n-r$.

In the red subgraph, a vertex $v\in A$ is isolated precisely when it has no neighbor in $R$ in the underlying graph $G$. 
This happens with probability $(1-p)^r$. 
Since $pr=\frac{f}{n}\,r=\frac{f e^{-h}}{2}=h/2$ and $p=o(1)$, we have
\begin{equation}
\mathbb P \Big(v \notin \bigcup_{u\in R}N(u) \Big)
=(1-p)^r
=\exp\bigl(-(1/2+o(1))h\bigr).
\end{equation}
Let $X$ denote the number of vertices in $A$ with no red neighbor.
Then $X\sim\mathrm{Bin}\bigl(n-r,(1-p)^r\bigr)$, and therefore
\begin{equation}
 \mu= 
\mathbb E X=(n-r)(1-p)^r=n\exp\bigl(-(1/2+o(1))h\bigr).
\end{equation}
Since $X$ is binomial, the Chernoff lower-tail bound gives
\begin{equation}
\mathbb P\left(X<\tfrac12\mu\right)
\leq \exp\left(-\frac{\mu}{8}\right)
=o(1),
\end{equation}
because $\mu\to\infty$. Moreover,
\begin{equation}
\frac{\mu}{r}
=
2\exp\bigl((1/2-o(1))h\bigr)\to\infty,
\end{equation}
and hence $\mu/2\geq r$ for all sufficiently large $n$. Therefore, with high probability,
\begin{equation}
X\geq \tfrac12\mu\geq r.
\end{equation}
These $X$ vertices are isolated in red, so every red component has size at most $n-r$. Since both colors yield components of size at most $n-r$, 
the largest monochromatic connected component has size at most 
$n-cne^{-h}=n-\tfrac{1}{2}n e^{-h}$ with high probability, where we choose $c=\frac12$.
\end{proof}

Let us now record a standard expansion anti-neighborhood estimate. For a set \(X\subseteq V(G)\), let $N(X)$ denote the neighborhood of $X$ in $V(G)\setminus X$, that is,
\begin{equation}
N(X):=\{\,v\in V(G)\setminus X:\ \exists\,x\in X\ \text{with}\ xv\in E(G)\,\}.
\end{equation}

\begin{lemma}\label{lem:anti-neigh}
Suppose $p=p(n)$ satisfies $p=o(1)$ and
$f=np=he^h=\omega(1)$, and set \(r:=\frac12 ne^{-h}\).
Let \(G\sim G(n,n,p)\) be the random bipartite graph with parts \((A,B)\).
Then, with high probability, every subset \(X\subseteq A\) (and likewise every \(X\subseteq B\))
of size \(r\) satisfies
\begin{equation}\label{eq:anti-neigh-bip-new}
|N(X)| \le n-2r.
\end{equation}
\end{lemma}

\begin{proof}
Fix \(X\subseteq A\) with \(|X|=r\).
For any \(b\in B\), the event \(b\notin N(X)\) holds with probability
\begin{equation}
(1-p)^r
=\exp\bigl(-pr+O(p^2r)\bigr)
=\exp\bigl(-(1/2+o(1))h\bigr),
\end{equation}
since \(pr=\dfrac{he^{h}}{n}\cdot \dfrac12 ne^{-h}=\tfrac12 h\).
Hence \(|B\setminus N(X)|\sim \mathrm{Bin}\big(n,(1-p)^r\big)\) and
\begin{equation}
\mu:=\mathbb{E}\big[|B\setminus N(X)|\big]= n e^{-(\tfrac12+o(1)) h}.
\end{equation}
The ``bad” event \(|N(X)|>n-2r\) is equivalent to \(|B\setminus N(X)|<2r=n e^{-h}\). Since \(\mu\gg 2r\) 
(indeed \(2r/\mu=e^{-h/2+o(h)}\)),
the deviation is linear in \(\mu\).
By the Chernoff lower-tail inequality,
\begin{equation}
\mathbb{P}\big(|N(X)|>n-2r\big)
=\mathbb{P}(|B\setminus N(X)|<2r)
\le \exp\!\Big(-\frac{(\mu-2r)^2}{2\mu}\Big)
=\exp\!\big(-\Omega(n e^{-h/2+o(h)})\big).
\end{equation}
Taking a union bound, and noting that 
$\binom{n}{r} \leq (\frac{en}{r})^r\leq \exp(r\log (en/r))\leq \exp(r(h+2))$. The 
probability that there exists a subset $X\subset A$ with $|X|=r$ and $|B\backslash N(X)|<2r$ is at most 
\begin{equation}
\binom{n}{r}\,\exp\!\big(-\Omega(n e^{-h/2+o(h)})\big) 
\le \exp\!\Big(r(h+2)-\Omega(n e^{-h/2+o(h)})\Big)=o(1),
\end{equation}
since \(r(h+2)=\tfrac12 n e^{-h}(h+2)=o(n e^{-h/2+o(h)})\) (equivalently \(e^{h/2}\gg h+1\)).
This proves \eqref{eq:anti-neigh-bip-new}.
\end{proof}

\begin{proof}[Proof of part (ii) of Theorem~\ref{thm:bipartite-mono}]
Assume \(f = np = h e^h\) and set \(r = \tfrac12 n e^{-h}\), and let \(G \sim G(n,n,p)\) with bipartition \((A,B)\).
Choose and fix \(X \subseteq A\) with \(|X|=r\). 
By Lemma~\ref{lem:anti-neigh}, with high probability, we have
\(|B\setminus N(X)|\ge 2r\), hence we can choose
\(Y \subseteq B \setminus N(X)\) with \(|Y|= r\) and therefore
\(E(X,Y)=\varnothing\) deterministically.

Put \(A_1 = A \setminus X\) and \(B_1 = B \setminus Y\).
Define the coloring as follows:
color each present edge between \(A_1\) and \(B_1\) blue,
and each present edge between \(X\) and \(B_1\) or between \(Y\) and \(A_1\) red.

Blue edges appear only between \(A_1\) and \(B_1\), so every blue component is contained in \(A_1 \cup B_1\); in particular the largest blue component has order at most
\begin{equation}
|A_1 \cup B_1| = 2n - 2r.
\end{equation}

Every red edge ends in either \(X\) or \(Y\). For any red connected component, namely \(C_X\) which intersects \(X\), since we know that \(|N(X)|\leq n-2r\) due to Lemma~\ref{lem:anti-neigh}, it follows that
\begin{equation}
|C_X| \leq  |X|+|N(X)| \le r + (n-2r) = n-r.
\end{equation}
By symmetry, we conclude that any red connected component \(C_Y\) which intersects
\(Y\) satisfies that
\(|C_Y|\leq n-r\) also.

Moreover, if a red component intersects \(X\), then it is contained in
\(X\cup B_1\), so together with any blue component it misses all vertices
of \(Y\). Similarly, if a red component intersects \(Y\), then together
with any blue component it misses all vertices of \(X\). Hence, one blue
component and one red component cover at most \(2n-r\) vertices. In conclusion, the following holds under the above coloring.
Any two blue monochromatic connected components
cover at most \(2n-2r\) vertices. Any two red components cover at most \(2n-2r\) vertices. One blue and one red component cover at most \(2n-r\) vertices.
The proof is completed.
\end{proof}

\begin{proof}[Proof of Theorem~\ref{thm:three-color} (ii)]
Fix a small constant $c_{0}\in(0,\log(5/4))$, set $c=c_{0}/6$,
and let $r=c_{0}/p$. For a fixed $X\subseteq V(G)$ with $|X|=r$, each vertex
$v\in V(G)\setminus X$ lies in $N(X)$ with probability
\begin{equation}
q
=1-(1-p)^r
=1-\exp\bigl(-pr+O(p^2r)\bigr)
=1-e^{-c_{0}}+o(1),
\end{equation}
since $pr=c_{0}$ and $p^2r=c_{0}p=o(1)$. Hence
\begin{equation}
|N(X)|\sim\mathrm{Bin}(n-r,q),
\qquad
\mu:=\mathbb E|N(X)|
=(1-e^{-c_{0}}+o(1))n.
\end{equation}

Since $c_{0}<\log(5/4)$, we have $1-e^{-c_{0}}<1/5$. Thus, for some
constant $\varepsilon>0$ and all sufficiently large $n$,
\begin{equation}
\mu\leq \left(\frac15-\varepsilon\right)n.
\end{equation}
Therefore, by the Chernoff upper-tail bound,
\begin{equation}
\mathbb P\left(|N(X)|\geq\frac n5\right)
\leq \exp(-\Omega(n)).
\end{equation}

The number of possible sets $X$ is at most
\begin{equation}
\binom nr
\leq
\exp\left(r\log\frac{en}{r}\right)
=
\exp\left(
\frac{c_{0}}{p}
\log\left(\frac{enp}{c_{0}}\right)
\right)
=
\exp(o(n)),
\end{equation}
since
\begin{equation}
\frac{\log(np)}{np}=o(1).
\end{equation}
Hence, by the union bound, with high probability every set
$X\subseteq V(G)$ of size $r$ satisfies
\begin{equation}
|N(X)|<\frac n5.
\end{equation}

\smallskip
By the standard elementary lower bound for the independence number of $G(n,p)$, with high probability $G$ contains an independent set 
$X$ of size $r$. Partition $X$ evenly into three subsets
$X_{1},X_{2},X_{3}$.

Split the remaining vertices into two equal parts $A$ and $B$, and further
decompose $A=A_{1}\cup A_{2}$ and $B=B_{1}\cup B_{2}$ so that the four
parts have equal sizes, up to rounding. Since $r=o(n)$ and
$|N(X)|\leq n/5$, we can choose these partitions such that
$N(X)\subseteq A_{1}$.

We now define a $3$-edge-coloring of $G$ as follows:
assign color~1 to all edges inside $A$ and inside $B$;
assign color~2 to all edges in the bipartite graphs
$G[A_{1},B_1]$ and $G[A_{2},B_2]$;
assign color~3 to all edges in the bipartite graphs
$G[A_{1},B_{2}]$ and $G[A_{2},B_1]$.
Finally, color every edge incident to $X_i$ with color~$i$, for
$i=1,2,3$.

\smallskip
Before adding the vertices of \(X\), every monochromatic component has order
at most \((n-r)/2\). Since \(X\) is independent and \(N(X)\subseteq A_1\),
every color-\(i\) component meeting \(X_i\) is contained in \(X_i\) together
with \(A\), \(A_1\cup B_1\), or \(A_1\cup B_2\), according as \(i=1,2,3\).
Each corresponding block has order at most \((n-r)/2\), and
\(|X_i|=r/3\). Hence every monochromatic component in \(G\) has order at most
\begin{equation}
\frac{n-r}{2}+\frac{r}{3}
=\frac{n}{2}-\frac{c}{p},
\end{equation}
where \(c=c_0/6\), and thus~\eqref{eq:3color-upper} is established.
\end{proof}

\section{Existence of large monochromatic components}\label{large_components}

\begin{definition}[Splitting of a bipartite graph]\label{def:splitting}
Let $G$ be a bipartite graph with parts $(X,Y)$.
A \emph{splitting} of $G[X,Y]$ is a quadruple of subsets
\begin{equation}
X_1,X_2\subseteq X,\qquad Y_1,Y_2\subseteq Y,
\end{equation}
forming partitions $X=X_1\sqcup X_2$ and $Y=Y_1\sqcup Y_2$.
We call a \emph{crossing edge} for a splitting 
any edge in $E(X_1,Y_2)\cup E(X_2,Y_1)$.
A \emph{clean splitting} is a splitting that does not admit any crossing edges.
Among $X_1\cup Y_1$ and $X_2\cup Y_2$, we call the one of smaller cardinality the 
\emph{smaller block} (ties broken arbitrarily).
\end{definition}

The following proposition ensures no clean splitting exists between 
any pair of vertex subsets of size $s_0$. We define two parameters and we will frequently use them. 
 Define
\begin{equation}\label{s_0y_0def}
y_0 = 4
n e^{-h},  \qquad s_0=100 y_0.
\end{equation}

\begin{prop}\label{thm:no-split-h}
Let $p=p(n)$ be such that $np=f=he^h=\omega(1)$. Let $y_0$ and $s_0$ be given by~\ref{s_0y_0def}. 
Then with high probability  the following holds. 
\begin{enumerate}
   \item[(a)] \textbf{(Erd\H{o}s--R\'enyi model).} 
Let 
   $G\sim G(n,p)$. For any disjoint vertex subsets,
$X, Y$ with $|X|= |Y|= s_0$, 
the induced bipartite graph $G[X,Y]$ admits no clean splitting whose smaller block $S$ satisfies
\begin{equation}
|S|=\min\{|X_1\cup Y_1|,\;|X_2\cup Y_2|\}> y_0.
\end{equation}
   \item[(b)] \textbf{(Bipartite model).} 
Let 
   $G\sim G(n,n,p)$ on the vertex bipartition $A,B$. 
   For any  disjoint vertex subsets,
$X\subseteq A, Y\subseteq B$,  with $|X|= |Y|= s_0$, 
the induced bipartite graph $G[X,Y]$ admits no clean splitting whose smaller block $S$ satisfies
\begin{equation}
|S|=\min\{|X_1\cup Y_1|,\;|X_2\cup Y_2|\}> y_0.
\end{equation}
\end{enumerate}
\end{prop}

\begin{proof}
We prove (a), and the proof of (b) is essentially the same. Fix any disjoint \(X,Y\subseteq[n]\) with \(|X|=|Y|=s_0\).
Consider a clean splitting, if one exists, whose smaller block has size
$s\in[y_0,s_0]$. After interchanging the indices $1$ and $2$ if necessary,
we may assume that the smaller block is $X_1\cup Y_1$.
Let \(|X_1|=a\) and \(|Y_1|=s-a\)
(so \(|X_2|=s_0-a\), \(|Y_2|=s_0-(s-a)\)).
The absent crossing edges are those between \(X_1\)–\(Y_2\) and \(X_2\)–\(Y_1\), hence
\begin{equation}
C(s,a)=|X_1||Y_2|+|X_2||Y_1|=a(s_0-s+a)+(s_0-a)(s-a)=2a^2-2sa+s_0 s.
\end{equation}
Thus, an upper bound on the failure probability is
\begin{equation}\label{Failure_bound}
    {n\choose s_0} {n-s_0\choose s_0} \sum_{s=y_0}^{s_0} \sum_{a=0}^s {s_0\choose a} {s_0\choose s-a} (1-p)^{C(s,a)}.
\end{equation}
Our goal in the remaining proof is to organize and estimate the above sum.

Let us start with the quadratic function \(C(s,a)\). Note that $C(s,a)$
attains its minimum over \(a \in [0,s]\) at \(a = s/2\), and we write 
\begin{equation}
\min_{0 \le a \le s} C(s,a) \geq s_0 s - \frac{s^2}{2}.
\end{equation}
Using the bound \((1-p)^m\le \exp({-pm})\), we have then $(1-p)^{C(s,a)}\leq \exp\Big(-p (s_0s-s^2/2)\Big)$.

The following Vandermonde’s identity will be applied to the inner sum over $a$: 
\begin{equation}
\sum_{a=0}^{s}\binom{s_0}{a}\binom{s_0}{s-a}=\binom{2s_0}{s}.
\end{equation}

Combining the estimates so far, we obtain the following estimate for any $s$ fixed, 
\begin{equation}
\sum_{a=0}^{s}\binom{s_0}{a}\binom{s_0}{s-a}(1-p)^{C(s,a)}
\;\le\;
\binom{2s_0}{s}\exp\!\Big(-p\Big(s_0 s-\frac{s^{2}}{2}\Big)\Big).
\end{equation}

Next, since \(s_0=100y_0\), 
for 
any 
\(s\in[y_0,s_0]\), 
we have
${2s_0 \choose s} 
\ \le\
\exp\!\Big(s\log(\frac{e\cdot 2s_0}{s})\Big)
\ =\
\exp\!\big( O(s_0)
\big).$ 

Then let us write
$T_s=\exp\!\Big(-p(s_0 s-\frac{s^2}{2})\Big)$, and split the range 
\([[y_0,s_0]=
[y_0,\lfloor s_0/2\rfloor]
\cup
[\lfloor s_0/2\rfloor+1,s_0]\).

\noindent \textbf{\boldmath Lower half: $s\in [y_0,\lfloor s_0/2\rfloor]$.}
In this range, we have
\begin{equation}
\frac{T_{s+1}}{T_s}
=\exp\!\Big(-p\big(s_0-s-\tfrac12\big)\Big)
\le \exp\!\Big(-p\big(\tfrac{s_0}{2}-\tfrac12\big)\Big)
\le \tfrac12.
\end{equation}
This ratio means 
$\sum T_s$ in this range can be bounded by a geometric series, which in turn is  essentially 
bounded by its first term. More precisely, 
\begin{equation}
\sum_{s=y_0}^{\lfloor s_0/2\rfloor} {2s_0\choose s}
T_s
\le \exp(O(s_0))   T_{y_0} 
=\exp(O(s_0) -398\, h \, y_0).
\end{equation}

\noindent\textbf{\boldmath Upper half: $s\in[\lfloor s_0/2\rfloor+1,s_0]$}. Note that 
the function 
\(s \mapsto p(s_0 s-\tfrac{s^2}{2})\) increases in \(s\) in this range. Therefore, 
each term $T_s$ in this range is bounded by $T_{\lfloor s_0/2\rfloor}=\exp (-150 \, h\, s_0 )$. 
Therefore, 
\begin{equation}
\sum_{s=\lfloor s_0/2\rfloor+1}^{s_0} 
\binom{2s_0}{s}T_s
\le \lceil \tfrac{s_0}{2}\rceil\, 
\exp 
\Big( O(s_0) \Big) \, 
T_{\,\lfloor s_0/2\rfloor}
=
\exp \Big(O(\log (s_0)) + O(s_0)  - 150\,h\,s_0\Big).
\end{equation}
Since \(s_0=100\,y_0\), the exponent equals
\(-15000\,h\,y_0+100(\ln 4)\,y_0+O(\log y_0)\),
which is far smaller than that of the lower-half contribution 
\(\exp(-398\,h\,y_0)\).  
Thus the upper-half sum is negligible.

Finally, note that 
the number of ordered disjoint pairs \((X,Y)\) with \(|X|=|Y|=s_0\) satisfies
\begin{equation}
\binom{n}{s_0}\binom{n-s_0}{s_0}
\ \le\ \exp\!\Big(2s_0(h+O(1))\Big).
\end{equation}

Therefore, noting that 
$2\, s_0 \, h - 
398\, h\, y_0
=
-198 \, h\, y_0
\to
-\infty$ since $h\to \infty$, 
the summation in (\ref{Failure_bound})
is bounded by 
\begin{equation}
\exp\!\Big(2s_0(h+O(1))\Big)\cdot  
\exp\! \Big(O(s_0) -398 \,h\,y_0 \Big)
\ = \ o(1).
\end{equation}
The proof is completed. 
\end{proof}

\begin{lemma}\label{lem:edge-y0}
Suppose \(np=f=he^{h}=\omega(1)\), and set \(y_{0}=4n e^{-h}\) as in~\eqref{s_0y_0def}.  
Then, with high probability, the following holds.  
\begin{enumerate}
   \item For \(G\sim G(n,p)\), every two disjoint vertex sets \(X,Y\subseteq V(G)\) with \(|X|=|Y|=y_0\) contain at least one edge between them.  
   \item For \(G\sim G(n,n,p)\) with vertex bipartition \(A\cup B\), every \(X\subseteq A\) and \(Y\subseteq B\) with \(|X|=|Y|=y_0\) contain at least one edge between them.  
\end{enumerate}
\end{lemma}

\begin{proof}
The same proof works for both statements. Let us only give the proof of (b). Fix \(X\subseteq A\), \(Y\subseteq B\) with \(|X|=|Y|=y_{0}\).
The probability that \(G[X,Y]\) has no edge is \((1-p)^{y_{0}^{2}}\le \exp(-p y_{0}^{2})\).
Since \(p=he^{h}/n\) and \(y_{0}=4n e^{-h}\), we have \(p y_{0}^{2}=16h n e^{-h}\),
so for fixed \(X,Y\),
\(\mathbb{P}(E(X,Y)=\emptyset)\le \exp(-16h n e^{-h})\).

Moreover, \(\binom{n}{y_{0}}\le (e n / y_{0})^{y_{0}} = (e^{h+1}/4)^{y_{0}}\),
so \(\binom{n}{y_{0}}^{2} \le \exp(8h n e^{-h})\) for large \(h\).
By the union bound,
\(\mathbb P(\exists\,X,Y:\,E(X,Y)=\emptyset)\le 
\exp(-8h n e^{-h})=o(1)\),
since \(h n e^{-h}\to\infty\) under \(np=\omega(1)\).
Thus, with high probability, every pair of subsets of size \(y_{0}\)
admits at least one edge between them. 
\end{proof}

\begin{coro}\label{cor_on_large_component} 
Suppose $np=f=
he^h=\omega(1)$, 
and set $y_0,s_0$ as in (\ref{s_0y_0def}).
Then the following holds with high probability.
\begin{enumerate}[(a)]
    \item In the Erd\H{o}s--R\'enyi model $G\sim G(n,p)$, for every disjoint vertex subsets
    $X, Y$ with $|X|\geq s_0, |Y|\ge s_0$, the induced bipartite graph $G[X,Y]$ contains a connected component covering all but at most $3y_0$ vertices of $X\cup Y$.
    
   Moreover, for any two disjoint vertex subsets $A$ and $B$ with
\(|A|\geq 4s_0, |B|\ge 4s_0\), every red--blue coloring of the edges of
\(G[A,B]\) contains a monochromatic connected component \(C\) such that
\begin{equation}
|C\cap A|\ge \frac{|A|}{2}-5y_0,
\qquad
|C\cap B|\ge \frac{|B|}{2}-5y_0.
\end{equation}

\item In the bipartite 
model \(G\sim G(n,n,p)\) with fixed bipartition \(A\cup B\) (so \(|A|=|B|=n\)),
for every \(X\subseteq A\) and \(Y\subseteq B\) with \(|X|\ge s_0, |Y|\ge s_0\), the induced bipartite graph \(G[X,Y]\)
contains a connected component covering all but at most \(3y_0\) vertices of \(X\cup Y\).

Moreover, for every red--blue coloring of the edges of \(G[A,B]\), there exists a monochromatic connected component \(C\) such that
\begin{equation}
|C \cap A| \;\ge\; \frac{n}{2} - 5y_0,
\qquad
|C \cap B| \;\ge\; \frac{n}{2} - 5y_0.
\end{equation}
\end{enumerate}
\end{coro}

\begin{proof} We will only prove  (a), since proof of (b) is the essentially the same. 
Let us prove the first statement. 
Choose \(X_0\subseteq X\) and \(Y_0\subseteq Y\) with
\(|X_0|=|Y_0|=s_0\), and consider the bipartite graph \(G[X_0,Y_0]\).
We claim that \(G[X_0,Y_0]\) contains a connected component of size at
least \(2s_0-y_0\).

Suppose otherwise that every connected component has size less than
\(2s_0-y_0\). If some connected component \(C\) has size greater than
\(y_0\), then
\begin{equation}
\bigl|(X_0\cup Y_0)\setminus C\bigr|
=2s_0-|C|>y_0.
\end{equation}
Since there are no edges between \(C\) and its complement, these two
sets form a clean splitting whose smaller block has size greater than
\(y_0\), contradicting Proposition~\ref{thm:no-split-h}.

Otherwise, every connected component has size at most \(y_0\).
By taking connected components one at a time, we can find a union \(U\)
of connected components such that
\begin{equation}
y_0<|U|\leq 2y_0.
\end{equation}
Since
\begin{equation}
\bigl|(X_0\cup Y_0)\setminus U\bigr|
\geq 2s_0-2y_0>y_0,
\end{equation}
the sets \(U\) and \((X_0\cup Y_0)\setminus U\) again form a clean
splitting whose smaller block has size greater than \(y_0\), contradicting
Proposition~\ref{thm:no-split-h}. Therefore, \(G[X_0,Y_0]\) contains a
connected component \(C\) satisfying
\begin{equation}
|C|\geq 2s_0-y_0.
\end{equation}

It follows that
\begin{equation}
|X_0\cap C|\geq s_0-y_0,
\qquad
|Y_0\cap C|\geq s_0-y_0.
\end{equation}
Lemma~\ref{lem:edge-y0} implies that there cannot exist \(y_0\) vertices
in \(X\setminus X_0\) with no neighbor in \(Y_0\cap C\). Hence at most
\(y_0\) vertices of \(X\setminus X_0\) fail to send an edge to
\(Y_0\cap C\), while every other vertex of \(X\setminus X_0\) lies in
the same connected component as \(C\).

By symmetry, at most \(y_0\) vertices of \(Y\setminus Y_0\) have no
neighbor in \(X_0\cap C\), while every other vertex of
\(Y\setminus Y_0\) lies in the same connected component as \(C\).
Combining these bounds with
\begin{equation}
\bigl|(X_0\cup Y_0)\setminus C\bigr|\leq y_0,
\end{equation}
we conclude that the connected component of \(G[X,Y]\) containing \(C\)
covers all but at most \(3y_0\) vertices of \(X\cup Y\).

Moreover, fix two disjoint vertex sets \(A,B\subseteq V(G)\) with
\(|A|\geq |B|\geq 4s_0\), and consider an arbitrary red--blue edge
coloring of the bipartite graph \(G[A,B]\). We first have the following claim.\\

\noindent {\textbf{Claim.} In the bipartite graph $G[A,B]$,
there is a monochromatic component $C$, such that 
$\max \{ |C\cap A|,|C\cap B|\} \geq s_0$.
}
\begin{proof}[Proof of Claim]
Suppose for contradiction that every monochromatic component $C$ satisfies that 
$|C\cap A|<s_0$ and $|C\cap B|<s_0$. 
In particular, every monochromatic component has order 
strictly smaller than $2s_0$.

Let $C_1$ be the largest monochromatic (say, blue) connected component 
in $G[A,B]$.  
Define $C_2$ to be the largest blue connected component in the bipartite graph on the vertex set $V\setminus C_1$. 
Continue this process greedily: for each $i\ge2$, let $C_i$ be the largest blue connected component 
in the remaining bipartite graph induced by $V\setminus(C_1\cup\cdots\cup C_{i-1})$. 
Stop the process 
as soon as there are no blue edges left in the remaining bipartite graph, 
or when the union 
$U=C_1\cup\cdots\cup C_k$ 
first satisfies 
$|U\cap A|\ge s_0$ or $|U \cap B|\ge s_0$.

If the process stops because there are no blue edges left before the second stopping condition is met, then
the bipartite subgraph induced by $V\setminus U$ 
meets both parts $A$ and $B$ in at least $3s_0$ vertices, 
and in the corresponding induced bipartite graph all present edges are red. 
By 
the first statement, 
we know 
this bipartite subgraph therefore contains a red connected component covering at least 
$6s_0-3y_0>2s_0>|C_1|$, contradicting the choice of $C_1$ as a largest monochromatic component.

Therefore, the process must stop the first time the union $U$ intersects one of the parts in at least $s_0$ vertices. 
Assume $|U\cap A|\ge s_0$ (the other case is analogous). 
By the minimality of $k$, before adding $C_k$ the union
$C_1\cup\cdots\cup C_{k-1}$ meets each of $A$ and $B$ in fewer than
$s_0$ vertices. Since $|C_k\cap B|<s_0$, it follows that
\begin{equation}
|U\cap B|<2s_0.
\end{equation}
Consequently,
\begin{equation}
|B\setminus U|> |B|-2s_0\ge 2s_0.
\end{equation}
Consider the bipartite graph 
$G[U\cap A,\,B\setminus U]$. 
By the first statement, 
it contains a red connected component of order at least 
$3s_0-3y_0> 2s_0 > |C_1|$, again contradicting the maximality of $C_1$.
\end{proof}

Thus we can take a  maximal monochromatic (say blue) component $C_\ast$ which intersects one part with at least $s_0$ vertices. 
Let us
assume that
$|C_\ast\cap A|\ge s_0$ (the other case can be similarly dealt with).
If $|C_\ast\cap B|\geq \frac{|B|}{2}-\frac{3y_0}{2}$, then we have a blue component which intersects $B$ with 
at least $\frac{|B|}{2}-\frac{3y_0}{2}$ vertices. 

If 
$|C_\ast \cap B|< \frac{|B|}{2}-\frac{3y_0}{2}$, 
then $|B\setminus C_\ast|\ge 
\frac{|B|}{2}+\frac{3y_0}{2}\geq s_0$. 
Applying 
the first statement to the bipartite graph 
$G[C_\ast \cap A, B\backslash  C_\ast]$, 
we obtain a red component $R$, which intersects $B$ with at least 
$\frac{|B|}{2}-\frac {3y_0}{2}$ vertices. Moreover, $|R\cap A|\geq |C_\ast\cap A|-3y_0\geq s_0-3y_0.$

In either case, extend the monochromatic component under consideration to the
monochromatic connected component of $G[A,B]$ containing it,
and exchange the names of the colors if necessary. Denote the resulting
component by $C_\ast$. Thus, in either case, we may assume that $C_\ast$ is a
maximal blue component satisfying
\begin{equation}
|C_\ast\cap B|\geq \frac{|B|}{2}-\frac{3y_0}{2},
\qquad
|C_\ast\cap A|\geq s_0-3y_0.
\end{equation}
Therefore, all present edges in
$G[A\setminus C_\ast,C_\ast\cap B]$ have color red. If 
$|C_\ast\cap A| \geq \frac{
|A|}{2}-\frac{3y_0}{2}$ then we are done. 
If $|C_\ast\cap A|
< \frac{
|A|}{2}-\frac{3y_0}{2}$, then in $G[A\backslash C_\ast, C_\ast\cap B]$, by the first statement, we can find a red component $R$ such that 
$|R\cap A| \geq 
\frac{|A|}{2}+\frac{3y_0}{2}-3y_0
\geq \frac{|A|}{2}-2y_0$ and $|R\cap B|\geq 
\frac{|B|}{2}-\frac{3y_0}{2}-3y_0\geq \frac{|B|}{2}- 5y_0$, which is a monochromatic component as required. If $|C_\ast\cap A|
\ge \frac{
|A|}{2}-\frac{3y_0}{2}$, then $C_\ast$ is the monochromatic component as required. 
\end{proof}

\begin{proof}[Proof of (i) in Theorem~\ref{thm:mono-sharp}]
Let $G\sim G(n,p)$ with $np=f=he^{h}$, and define $y_0,s_0$ as in~\eqref{s_0y_0def}, and fix an arbitrary red--blue edge coloring. 
Let $A$ be a largest monochromatic connected component; without loss of generality $A$ is blue.\\

\noindent \smallskip
\textbf{Case 1: $|A|\le s_0$.}
Greedily take disjoint blue components $C_1,C_2,\dots$ until either there are no new blue components and we obtain $W$ as the union of all blue components, or the union $W=\bigcup_{i=1}^k C_i$ satisfies
$s_0\le |W|\le 2s_0$ for a minimal $k$ (no more than $2s_0$ is possible since each added component has size at most $s_0$). 
In the first case, we have $|W|<s_0$, and all present edges inside
$G[V(G)\setminus W]$ are red. Since $s_0=o(n)$, for all sufficiently large $n$ we have
$|V(G)\setminus W|\geq 2s_0$. Partition $V(G)\setminus W$ into two
sets of size at least $s_0$. Applying
Corollary~\ref{cor_on_large_component}\,(a) to the induced bipartite
graph between these two sets yields a red connected component of order
at least
\begin{equation}
n-|W|-3y_0>n-s_0-3y_0.
\end{equation}

In the second case, let $W^c=V\setminus W$. 
By the definition of the monochromatic components forming $W$, every
present edge in $G[W,W^c]$ has the red color.
Applying
Corollary~\ref{cor_on_large_component}\,(a) to the bipartite graph $G[W,W^{c}]$ yields a connected red component of order at least $n-3y_0$. This proves the claim in this case.\\

\noindent \smallskip
\textbf{Case 2: $|A|> s_0$.}
Set $A^c:=V(G)\setminus A$. If $|A^c|<s_0$, then $|A|>n-s_0$ and we are done.
Otherwise $|A^c|\ge s_0$. Again, every present edge in $G[A,A^c]$ is red, so
Corollary~\ref{cor_on_large_component} (a) applied to the bipartite graph $G[A,A^c]$ gives a red connected component of order at least $n-3y_0$.

\smallskip
In both cases there is a monochromatic connected component of order at least $n-s_0-3y_0$, proving (i).
\end{proof}

\begin{proof}[Proof of (i) of Theorem~\ref{thm:bipartite-mono}]
Let $G\sim G(n,n,p)$ with $f=np=h e^{h}$, and set $y_0,s_0$ as in~(\ref{s_0y_0def}). 
Then fix an arbitrary red-blue coloring. Our goal is to prove that two connected components cover all but $2s_0$ vertices. 

By Corollary~\ref{cor_on_large_component} (b), we can choose a monochromatic connected component $C_1$. Without loss of generality, assume that $C_1$ is blue. We have:
\begin{equation}\label{eq:balance-C1}
|C_1\cap A| \geq  \frac n2-5 y_0, \, \, |C_1\cap B|\ \ge\ \frac{n}{2}-5y_0.
\end{equation}

Now we define 
\begin{equation}
A_1 := A \cap C_1, \qquad 
X := A \setminus A_1, \qquad
B_1 := B \cap C_1, \qquad
Y := B \setminus B_1.
\end{equation}

\noindent{\textbf{Case 1. Both $|X|<s_0$ and $|Y|<s_0$.}}
In this case, the single blue component $C_1$ already covers at least $2n-2s_0$ vertices, 
so the desired conclusion holds.\\
\noindent{\textbf{Case 2. $|X|\ge s_0$ and $|Y|<s_0$ (or, by symmetry, $|X|<s_0$ and $|Y|\ge s_0$).}}
All present edge
in the bipartite graph 
$G[X,B_1]$ must be red. 
Applying Corollary~\ref{cor_on_large_component} (b) 
to this bipartite graph yields a red connected component $R$ 
that covers all but at most $3y_0$ vertices of $X\cup B_1$. 
Hence the two monochromatic components $C_1$ and $R$ together cover all but at most $s_0+3y_0$ vertices of $G$, 
and the desired conclusion follows.\\
\noindent{\textbf{Case 3. Both $|X|\ge s_0$ and $|Y|\ge s_0$.}} Again we consider the bipartite graph induced on the two parts $X$ and $B_1$. Applying Corollary~\ref{cor_on_large_component} (b)
to this bipartite graph yields a red connected component $R$ 
that covers all but at most $3y_0$ vertices of $X\cup B_1$. Similarly, for the bipartite graph $G[A_1,Y]$,
we apply  Corollary~\ref{cor_on_large_component} (b)
again to 
obtain 
a red connected component $R'$ 
that covers all but at most $3y_0$ vertices of $Y\cup A_1$. Therefore, the two (red) monochromatic connected components $R$ and $R'$ together cover all but at most $6y_0$ vertices of $G$, and the desired conclusion follows. 
\end{proof}

\section{Proof of part \textnormal{(i)}
 of Theorem~\ref{thm:three-color}}\label{ioftheorem3}
This section is devoted to the proof of Theorem~\ref{thm:three-color}, part (i).  
Let \(G \sim G(n,p)\) with \(np = \omega(1)\) and $p=o(1)$. We also write \(f = h e^{h}\), and define
\begin{equation}
y_0 = 4n e^{-h}, \qquad s_0 = 100 y_0.
\end{equation}
Fix an absolute constant \(C > 0\), whose precise value will be determined only at the end of the proof.

\medskip
\noindent\textbf{Standing assumption.}
For the sake of contradiction, assume that there exists a 3-edge-coloring of \(E(G)\), with colors red, blue, and green, such that every monochromatic connected component has order strictly smaller than
\(\frac{n}{2}-\frac{C}{p}\).

\medskip
Under this assumption, we shall gradually construct a sequence of large monochromatic subgraphs.
At each stage, the construction will take place within a bipartite subgraph of \(G\),
and we will repeatedly apply Corollary~\ref{cor_on_large_component}
to find new large components while keeping careful track of the color constraints.
This process will eventually lead to a global three–color splitting structure,
which in turn contradicts the standing assumption and completes the proof of
Theorem~\ref{thm:three-color}\,(i). Let us state this structure clearly as the following lemma, which is a sort of a stability result.

\begin{lemma}[Global three–color splitting]\label{lem:global-splitting} Under the standing assumption, 
with high probability, there exist pairwise disjoint vertex subsets
\begin{equation}
A_1, A_2, B_1, B_2 \subseteq V(G), 
\end{equation}
where \(|A_i|, |B_i| = \tfrac{n}{4} + o(n)\) for each \(i = 1,2\), with the following coloring restrictions (after possibly
re-naming the colors):
\begin{enumerate}
\item Set $A=A_1\cup A_2,B=B_1\cup B_2$.
Then \(A\) lies entirely inside one green connected component, while \(B\) lies
entirely inside a \emph{different} green connected component.
\item \(A_{1}\cup B_{1}\) lies inside one blue connected component, and
\(A_{2}\cup B_{2}\) lies inside a \emph{different} blue connected component;
\item \(A_{1}\cup B_{2}\) lies inside one red connected component, and
\(A_{2}\cup B_{1}\) lies inside a \emph{different} red connected component.
\end{enumerate}
\end{lemma}

\begin{proof}
We will apply 
Corollary~\ref{cor_on_large_component}
several times below. 
For clarity of exposition, the argument is divided into several steps.\\

\noindent\textbf{Step 1.}
There exists a red connected component \(C_{\text{red}}\) with \(|C_{\text{red}}| = \tfrac{n}{2} - o(n)\).

\smallskip
Among the three colors, choose one (say blue)
which meets the maximum number of vertices.
By the pigeonhole principle, and since for our choice of \(p\) the random graph \(G(n,p)\) 
has a unique giant component of order \(n-o(n)\),
the blue color must cover at least \(n/3 - o(n)\) vertices in total.   
If some blue connected
component already has order at least \(4s_0\), take \(C_{\text{blue}}^1\) to be that component. Here, the superscript '1' indicates that this is the first blue component (union) we select; later we will identify additional blue components until we reach the one that satisfies our requirements. If all blue
components have order at most \(4s_0\), but since blue edge cover much more than \(4s_0\) vertices,
we may form \(C_{\text{blue}}^1\) as a union of blue
components, 
putting together blue components one by one until
$ 4s_0 \le |C_{\text{blue}}^1| \le 8s_0.$
By construction, there is no blue
edge between \(C_{\text{blue}}^1\) and \(V\setminus C_{\text{blue}}^1\), so the bipartite graph
\(G[C_{\text{blue}}^1, V\setminus C_{\text{blue}}^1]\) is colored only with red and green.
Applying Corollary~\ref{cor_on_large_component} (a), we conclude 
that this bipartite subgraph contains a monochromatic (say red)
component 
\(C_{\text{red}}\) such that
\begin{equation}
|C_{\text{red}}\cap C_{\text{blue}}^1| \ge \frac{|C_{\text{blue}}^1|}{2} - 5y_0, \qquad
|C_{\text{red}}\cap (V\setminus C_{\text{blue}}^1)| \ge \frac{|V\setminus C_{\text{blue}}^1|}{2} - 5y_0.
\end{equation}
Consequently, 
\begin{equation}\label{red_n/2}
|C_{\text{red}}| \ge \frac{n}{2} - 10y_0.
\end{equation}Extend $C_{\mathrm{red}}$ to the red connected component of $G$
containing it, and continue to denote this component by
$C_{\mathrm{red}}$. Its order can only increase, so by~(\ref{red_n/2})
and the standing assumption, we still have
\begin{equation}
|C_{\mathrm{red}}|=\frac{n}{2}-o(n).
\end{equation}
Since $C_{\mathrm{red}}$ is now a maximal red connected component of
$G$, there are no red edges between $C_{\mathrm{red}}$ and
$V(G)\setminus C_{\mathrm{red}}$.
\\

\noindent {\textbf{Step 2}. There is one large blue component in \(G[C_{\text{red}},\,V\setminus C_{\text{red}}]\).}

Since there are no 
red edges between \(C_{\text{red}}\) and \(V\setminus C_{\text{red}}\), all edges of
\(G[C_{\text{red}},V\setminus C_{\text{red}}]\) are colored blue or green. Applying Corollary~\ref{cor_on_large_component} (a)
to
this bipartite graph gives a monochromatic (say blue) component \(C_{\text{blue}}^2\) such that
\begin{equation}
|C_{\text{blue}}^2\cap C_{\text{red}}|\ge \frac{|C_{\text{red}}|}{2}-5y_0,\qquad |C_{\text{blue}}^2\cap (V\setminus C_{\text{red}})|\ge \frac{|V\setminus C_{\text{red}}|}{2}-5y_0.
\end{equation}
Extend $C_{\mathrm{blue}}^2$ to the blue connected component of $G$
containing it, and continue to denote this component by
$C_{\mathrm{blue}}^2$. By the lower bounds above and the standing
assumption, we still have
\begin{equation}
|C_{\mathrm{blue}}^2|=\frac{n}{2}-o(n).
\end{equation}
Now define
\begin{equation}
X_1:=C_{\text{blue}}^2\cap C_{\text{red}},\qquad Y_1:=C_{\text{blue}}^2\cap (V\setminus C_{\text{red}}).
\end{equation}
It follows that \(|X_1|,|Y_1|=\tfrac{n}{4}+o(n)\). Put \(Y_2:=C_{\text{red}}\setminus X_1\) and \(X_2:=(V\setminus C_{\text{red}})\setminus Y_1\);
therefore \(|X_2|,|Y_2|=\tfrac{n}{4}-  o(n)\).\\

\noindent\textbf{Step 3}. There are two large green components inside \(G[C_{\text{red}}, V\setminus C_{\text{red}}]\).

Since \(C_{\text{blue}}^2\) is a maximal blue component across the cut,
there are no blue edges between \(Y_1\) and \(Y_2\), and none between \(X_1\) and \(X_2\).
Hence the bipartite graphs \(G[Y_1,Y_2]\) and \(G[X_1,X_2]\) contain only green edges.
Applying Corollary~\ref{cor_on_large_component} (a)
to \(G[Y_1,Y_2]\),
we obtain a large green component \(A\) satisfying
\begin{equation}
A \subseteq Y_1 \cup Y_2, \qquad |A| = \tfrac{n}{2} - o(n).
\end{equation}
In particular, \(A\) intersects both \(Y_1\) and \(Y_2\) in all but $o(n)$ vertices. 
To simplify notation, from now on we update the definitions:
\begin{equation}
Y_1 := A \cap Y_1, \qquad Y_2 := A \cap Y_2,
\end{equation}
so that \(A = Y_1 \cup Y_2\) denotes precisely this green component,
and the updated sets \(Y_1\) and \(Y_2\) refer to its two sides in the bipartition. Clearly, it still holds that $|Y_1|,|Y_2|=n/4-o(n)$.
An entirely analogous application of Corollary~\ref{cor_on_large_component} (a)
produces another large green component \(B \subseteq X_1 \cup X_2\);
we redefine \(X_i := B \cap X_i\) for \(i=1,2\), so that \(B = X_1 \cup X_2\).
Moreover, 
$|X_1|, |X_2|=n/4-o(n)$.
Since $|A|+|B|=n-o(n)$, the sets $A$ and $B$ must lie in distinct green connected components, and hence there are no green edges between them.

\medskip
\noindent\textbf{Step 4.} There is a red--blue splitting across \(G[A,B]\).

By construction there are no green edges between \(A\) and \(B\), so the bipartite graph \(G[A,B]\) is colored only red and blue. 
The idea is to repeat Step~2 and Step~3 inside \(G[A,B]\), obtaining the corresponding blue--red splittings.

Apply Corollary~\ref{cor_on_large_component}(a) to find a blue component \(C_{\text{blue}}^3\), intersecting both \(A\) and \(B\) in \(\frac n4-o(n)\) vertices. Extend \(C_{\text{blue}}^3\) to the blue connected component of \(G\) containing it, and continue to denote this component by \(C_{\text{blue}}^3\). Set
\begin{equation}
A_{1}=A\cap C_{\text{blue}}^3, \qquad B_{1}=B\cap C_{\text{blue}}^3.
\end{equation}
Define
\begin{equation}
A_{2}=A\setminus A_{1}, \qquad B_{2}=B\setminus B_{1}.
\end{equation}

Since \(C_{\text{blue}}^3\) is a blue connected component of \(G\), there are no blue edges between \(A_{1}\) and \(B_{2}\), and none between \(A_{2}\) and \(B_{1}\). Since there are also no green edges between \(A\) and \(B\), the bipartite graphs \(G[A_{1},B_{2}]\) and \(G[A_{2},B_{1}]\) contain only red edges.

Next apply Corollary~\ref{cor_on_large_component}(a) to the bipartite graphs \(G[A_{1},B_{2}]\) and \(G[A_{2},B_{1}]\). This yields two large red connected components, each of order \(\tfrac n2-o(n)\). Extend the two red components obtained above to red connected components of \(G\). They are distinct, since otherwise one red connected component would contain \(n-o(n)\) vertices.

At each application of Corollary~\ref{cor_on_large_component}(a), we discard a leftover set of size \(o(n)\). After updating the definitions of \(A_{1},A_{2},B_{1},B_{2}\), we obtain that
\begin{equation}
|A_{1}|, |A_{2}|, |B_{1}|, |B_{2}|=\frac n4-o(n).
\end{equation}
Since the two red connected components are distinct, there are no red edges between the updated sets \(A_{2}\) and \(B_{2}\).

Finally consider the bipartite graph \(G[A_{2},B_{2}]\). There are no green edges between \(A_{2}\) and \(B_{2}\), since \(A_{2}\subseteq A\) and \(B_{2}\subseteq B\), and there are no red edges between them by the preceding paragraph. Hence all edges currently present in this graph are blue. A further application of Corollary~\ref{cor_on_large_component}(a) yields a large blue connected component, also of order \(\frac n2-o(n)\). We update the definitions of \(A_{2}\) and \(B_{2}\) once more by discarding another \(o(n)\) set, and extend this blue component to the blue connected component of \(G\) containing it.

This blue connected component is distinct from \(C_{\text{blue}}^3\), since otherwise one blue connected component would contain \(A_{1}\cup B_{1}\cup A_{2}\cup B_{2}\), and hence would have order \(n-o(n)\).

Finally, redefine
\begin{equation}
A=A_{1}\cup A_{2}, \qquad B=B_{1}\cup B_{2}.
\end{equation}
After updating the definitions, \(A\) is contained inside the green connected component obtained in Step~3, while \(B\) is contained inside the other green connected component. Moreover, \(A_{1}\cup B_{1}\) and \(A_{2}\cup B_{2}\) lie inside distinct blue connected components, while \(A_{1}\cup B_{2}\) and \(A_{2}\cup B_{1}\) lie inside distinct red connected components. Thus all statements follow, completing the proof.
\end{proof}

The next standard 
lemma is used for us to determine the constant $C$.

\begin{lemma}\label{lem:expansion-3quarters}
Let \(G\sim G(n,p)\) with \(np=\omega(1)\) and $p=o(1)$.
Fix any constant \(K>\log 4\).
Then, with high probability, every subset \(X\subseteq V(G)\) with \(|X|=\frac{K}{p}\) satisfies
\begin{equation}
|N(X)|\ \ge\ \Big(\frac{3}{4}+\varepsilon\Big)n,
\end{equation}
where \(\varepsilon>0\) is a small constant depending only on \(K\).
\end{lemma}

\begin{proof}
Fix \(K>\log 4\) and choose a small constant \(\varepsilon>0\) such that  
\begin{equation}\label{eq:eps-choice}
1-e^{-K}\ \ge\ \frac{3}{4}+2\varepsilon .
\end{equation}
Fix any subset \(X\subseteq V(G)\) with \(|X|=\frac{K}{p}\).
For each vertex \(v\in V(G)\setminus X\), let
\begin{equation}
q:= \mathbb P(v\in N(X))=1-(1-p)^{\frac{K}{p}}\ \ge\ 1-e^{-K}\ \ge\ \frac{3}{4}+2\varepsilon.
\end{equation}
Hence \(|N(X)|\) 
has the binomial distribution  \(\mathrm{Bin}(n-\frac{K}{p},q)\),
with \(q\ge\frac{3}{4}+2\varepsilon\),
so that
\begin{equation}\label{eq:mu}
\mathbb{E}|N(X)|
=
q\left(n-\frac{K}{p}\right)
\geq
\left(\frac{3}{4}+\frac{3}{2}\varepsilon\right)n
\end{equation}
for all large \(n\), since \(K/p=o(n)\).

By Chernoff’s inequality, there exists \(c=c(K,\varepsilon)>0\) such that
\begin{equation}\label{eq:chernoff}
\mathbb P\!\left(|N(X)|<\Big(\frac{3}{4}+\varepsilon\Big)n\right)\le e^{-c n}.
\end{equation}
As
\begin{equation}
\binom{n}{\frac{K}{p}}\le \exp\!\Big(\frac{K}{p}\log\frac{enp}{K}\Big)
=\exp(o(n)),
\end{equation}
a union bound gives
\begin{equation}
\mathbb P\!\left(\exists\,X:\ |X|=\frac{K}{p},\ |N(X)|<\Big(\frac{3}{4}+\varepsilon\Big)n\right)
\le e^{o(n)}\cdot e^{-c n}=o(1).
\end{equation}
Thus, with high probability, every subset \(X\) of size \(\frac{K}{p}\) satisfies
\(|N(X)|\ge(\frac{3}{4}+\varepsilon)n\), as claimed.
\end{proof}

\begin{proof}[End of proof of (i) of Theorem~\ref{thm:three-color}]
Let the constant \(K>\log 4\) be given by Lemma~\ref{lem:expansion-3quarters}, 
and set \(C := 2K\).
Applying Lemma~\ref{lem:global-splitting}, we obtain the global three–color
splitting structure: there exist disjoint vertex subsets
\begin{equation}
A = A_{1}\cup A_{2},\qquad B = B_{1}\cup B_{2},
\end{equation}
with \(|A_{1}|,|A_{2}|,|B_{1}|,|B_{2}| = \frac{n}{4}-o(n)\), such that
\(A\) lies entirely inside one green connected component, while \(B\) lies
entirely inside a \emph{different} green connected component;
\(A_{1}\cup B_{1}\) lies inside one blue connected component, and
\(A_{2}\cup B_{2}\) lies inside a \emph{different} blue connected component;
\(A_{1}\cup B_{2}\) lies inside one red connected component, and
\(A_{2}\cup B_{1}\) lies inside a \emph{different} red connected component.

Let \(R = V(G)\setminus(A\cup B)\), so that \(|R| = o(n)\).
If 
\begin{equation}
|R| < \frac{4K}{p},
\end{equation}
then \(|A| + |B| \ge n - \frac{4K}{p}\),
and hence one of \(A\) or \(B\) has order 
at least
\(n/2 - \frac{2K}{p}\), implying that at least one green connected component has order no less than $\frac n2 -\frac{2K}{p}$,
contradicting the \textbf{standing assumption}. 
Thus we may assume 
\begin{equation}
|R| \ge \frac{4K}{p}.
\end{equation}

By Lemma~\ref{lem:expansion-3quarters}, the number of vertices in \(R\) whose neighborhood fails to intersect \(A_1\) is at most \(\frac{K}{p}\). Indeed, suppose that there were at least \(\frac{K}{p}\) such vertices, and choose a subset \(X\) of exactly \(\frac{K}{p}\) of them. Then
\begin{equation}
N(X)\cap A_1=\varnothing,
\end{equation}
and therefore
\begin{equation}
|N(X)|\leq n-|A_1|
=\frac{3n}{4}+o(n)
<
\left(\frac{3}{4}+\varepsilon\right)n
\end{equation}
for all sufficiently large \(n\), contradicting Lemma~\ref{lem:expansion-3quarters}. The same argument applies to each of \(A_2,B_1,B_2\).

Hence the total number of vertices in \(R\) whose neighborhood misses at least one of the four parts \(A_1,A_2,B_1,B_2\) is at most \(\frac{4K}{p}\). Let \(R'\) denote the set of vertices having a neighbor in each of the four parts. Then
\begin{equation}
|R'|\geq |R|-\frac{4K}{p}.
\end{equation}

For each \(x \in R'\), the colors of its incident edges must satisfy
the following global splitting constraints:

\begin{itemize}
\item[\textbf{(G)}] \textbf{Green rule.}  
If \(x\) is incident to a green edge to \(A_1\) or \(A_2\),
then all edges from \(x\) to \(B_1\) and \(B_2\) are non-green;
otherwise, the vertex subsets \(A\) and \(B\) lie in a single green
connected component, a contradiction. Similarly, if \(x\) is incident
to a green edge to \(B_1\) or \(B_2\), then all edges from \(x\) to
\(A_1\) and \(A_2\) are non-green.
   \item[\textbf{(B)}] \textbf{Blue rule.}  
If \(x\) is incident to a blue edge to \(A_1\) or \(B_1\),
then all edges from \(x\) to \(A_2\) and \(B_2\) are non-blue;
otherwise, the vertex subsets \(A_1,B_1,A_2,B_2\) would all belong
to a single blue connected component, a contradiction.
Similarly, if \(x\) is incident to a blue edge to \(A_2\) or \(B_2\),
then all edges from \(x\) to \(A_1\) and \(B_1\) are non-blue.

\item[\textbf{(R)}] \textbf{Red rule.}  
If \(x\) has a red edge to \(A_1\) or \(B_2\),
then all edges from \(x\) to \(A_2\) and \(B_1\) are non-red;
otherwise, the vertex subsets \(A_1,B_2,A_2,B_1\) would all belong
to a single red connected component, a contradiction.
Similarly, if \(x\) has a red edge to \(A_2\) or \(B_1\),
then all edges from \(x\) to \(A_1\) and \(B_2\) are non-red.
\end{itemize}

Because every vertex \(x\in R'\) has neighbors in each of the four regions,
it is impossible for \(x\) to realize all these adjacencies using only red
and blue edges: by the blue and red rules, each of these colors can meet
vertices in at most one of its two opposite pairs, and the union of one
blue pair and one red pair contains at most three of the four regions.
Thus \(x\) must have a green edge to at least one of the four regions.
By rule~\textbf{(G)}, all green edges incident to \(x\) go towards exactly
one side, either \(A=A_1\cup A_2\) or \(B=B_1\cup B_2\), but not both.

Define
\begin{align}
R_A :&= \{x \in R' : \text{$x$ has a green neighbor in $A$}\},\\
R_B :&= \{x \in R' : \text{$x$ has a green neighbor in $B$}\}.
\end{align}
Then it follows that \(R_A \cup R_B = R'\) and \(R_A \cap R_B = \emptyset\). Let $A'=A\cup R_A$ and $B'=B\cup R_B$. 
Note that 
the sum of the sizes is 
\begin{equation}
    |A'|+|B'|\geq n-|R|+|R|-\frac{4K}{p}=n-\frac{4K}{p}.
\end{equation} By pigeonhole principle, one of the sets $A'$ and $B'$ has size at least $\frac{n}{2}-\frac{2K}{p}=\frac{n}{2}-\frac{C}{p}$, where we have taken $C=2K$. Note that the resulting set belongs to a single green connected component, and 
this contradicts the \textbf{standing assumption} that all monochromatic components
have order strictly smaller than 
\(\frac n2 - \frac{C}{p}\).

Hence, together with the argument for \(|R| < \frac{4K}{p}\),
the proof of part (i) of Theorem~\ref{thm:three-color} is completed.\end{proof}

\bibliographystyle{plain}
\addcontentsline{toc}{chapter}{Bibliography}
\bibliography{Gnpcombined}
\end{document}